\documentclass[11pt]{amsart}

\usepackage[usenames,dvipsnames,svgnames,table]{xcolor} %pour les couleurs
\usepackage{amssymb,amsfonts,amsrefs}
\usepackage{amsmath,amsthm}
\usepackage[active]{srcltx}		%reversed
\usepackage{pdfsync}					%reversed
\newcommand{\dd}{\,{\rm d}}

\newcommand\R{{\mathbb{R}}}
\renewcommand\P{{\mathbb{P}}}

\renewcommand\div{{\rm div}}

\newtheorem{theorem}{Theorem}[section]
\newtheorem{proposition}[theorem]{Proposition}
\newtheorem{lemma}[theorem]{Lemma}

\theoremstyle{definition}

\theoremstyle{remark}
\newtheorem{remark}[theorem]{Remark}
\numberwithin{equation}{section}

\begin{document}

\title[(\today)]
{Uniqueness Theorems for the Boussinesq System}

\author{Lorenzo Brandolese}
 
 \address{L. Brandolese: Universit\'e de Lyon, Universit\'e Lyon 1.
 CNRS - Institut Camille Jordan,
 43 bd. du 11 novembre,
 Villeurbanne Cedex F-69622, France.}
 \email{brandolese{@}math.univ-lyon1.fr}
\urladdr{http://math.univ-lyon1.fr/$\sim$brandolese}

\author{Jiao He}

\address{J. He Universit\'e de Lyon~; Universit\'e Lyon 1~;
 CNRS - Institut Camille Jordan,
 43 bd. du 11 novembre,
 Villeurbanne Cedex F-69622, France.}
 \email{he{@}math.univ-lyon1.fr}


\date{\today - Lorenzo Brandolese, Jiao He}

\begin{abstract}
We address the uniqueness problem for mild solutions of the Boussinesq system in $\R^3$. We provide several uniqueness classes on the velocity and the temperature, generalizing in this way the classical $C([0,T]; L^3(\R^3))$-uniqueness result for mild solutions of the Navier-Stokes equations.
\end{abstract}


\maketitle

%%%%%%%%%%%%%%%%%%%%%%%%%%%%%%%
%%%%%%%%%%%%%%%%%%%%%%%%%%%%%%%
%%%%%%%%%%%%%%%%%%%%%%%%%%%%%%%

\section{Introduction}
\label{sec:intro}

The incompressible Boussinesq system describes the dynamics of a viscous incompressble fluid with heat exchanges. This system arises from an approximation on a system coupling the classical Navier-Stokes equations and the equations of thermodynamics. In this approximation, the variations of the density due to heat transfers are neglected in the continuity equation, but are taken into account in the equation of the motion through an additional buoyancy term proportional to the temperature variations.

This paper deals with the uniqueness problems for mild solutions of the Boussinesq system.
With a minor loss of generality, and just to simplify the presentation, we will assume in the sequel that the physical constant are all equal to one.
In this case, the Boussinesq system can be written as the following form,  

\begin{equation}
\label{B} 
\left\{
\begin{aligned}
 &\partial_t \theta +u \cdot \nabla \theta =  \Delta \theta\\
 &\partial_t u +u\cdot\nabla u+\nabla p=\Delta u+\theta e_3\\
  &\nabla\cdot u=0\\
 &u|_{t=0}=u_0,\;\,\theta|_{t=0}=\theta_0.
\end{aligned}
\right.
\qquad x\in \R^3, t\in \R_{+}
\end{equation}
Here $u\colon\R^3\times\R^+\to \R^3$ is the velocity field. The scalar fields 
$p\colon \R^3\times\R^+\to \R$ and $\theta\colon\R^3\times\R^+\to\R$ denote respectively 
the pressure and the  temperature of the fluid. Moreover, $e_3=(0,0,1)$ is the unit vertical vector.

In the case $\theta\equiv0$, this system reduces to the classical Navier--Stokes equations.

The integral formulation of the Boussinesq system reads:
\begin{equation}
 \label{IE}
\left\{
\begin{aligned}
 &\theta(t)=e^{t\Delta}\theta_0-\int_0^t e^{(t-s)\Delta}\nabla \cdot(\theta u)(s)\dd s\\
 &u(t)=e^{t\Delta}u_0-\int_0^t e^{(t-s)\Delta}\P\nabla \cdot(u\otimes u)(s)\dd s
+\int_0^t e^{(t-s)\Delta}\P \theta(s) e_3\dd s.\\
&\nabla \cdot u_0=0
\end{aligned}
\right.
\end{equation}
Here $\mathbb{P}$ denotes the projector on the space of divergence-free fields, which is also called Leray’s projector.
In this paper, we will work directly with the integral form \eqref{IE} rather than the original system \eqref{B}. The solutions of \eqref{IE} are usually called \emph{mild solutions}. The equivalence between the two systems is not only formal, but can be established rigorously in quite general functional settings. We refer to the book of Lemari{\'e}-Rieusset (see Theorem 1.2 in \cite{Lem02}), for this issue in the particular case of the Navier--Stokes equations.

To write our system in a more compact form,
we can replace the equation of~$\theta$ inside 
the last integral and we get
\[
\int_0^t e^{(t-s)\Delta}\P \theta(s) e_3\dd s
=t\,e^{t\Delta}\P \theta_0e_3
-\int_0^t e^{(t-s)\Delta}(t-s)\P\nabla\cdot(\theta u)(s)\dd s\,\,e_3.
\]
Next, let us introduce the three bilinear maps
\begin{subequations}
\begin{align}
\label{B1}
 &B_1(u,\tilde u)=-\int_0^t e^{(t-s)\Delta}\P\nabla\cdot(u\otimes \tilde u)(s)\dd s,\\
 \label{B2}
 &B_2(u,\theta)=-\biggl(\int_0^t e^{(t-s)\Delta}(t-s)\P\nabla\cdot(u\theta)(s)\dd s\biggr)e_3,\\
 \label{B3}
 &B_3(u,\theta)=-\int_0^t e^{(t-s)\Delta}\nabla\cdot(u\theta )(s)\dd s,
\end{align}
\end{subequations}

Then our system~\eqref{IE} can be rewritten as

\begin{equation}
 \label{IEE}
\left\{
\begin{aligned}
 &u(t)=e^{t\Delta}[u_0+ t\P\theta_0 e_3]+B_1(u,u)+B_2(u,\theta),\\
 &\theta(t)=e^{t\Delta}\theta_0+B_3(u,\theta)\\
 &\nabla \cdot u_0=0
\end{aligned}
\right.
\end{equation}
This system  is left invariant by the natural scaling $(u,\theta)\mapsto (u_\lambda,\theta_\lambda)$,
with $\lambda>0$ and
\[
u_\lambda(x,t)=\lambda u(\lambda x,\lambda^2 t),\qquad\text{and}\qquad 
\theta_\lambda=\lambda^3\theta(\lambda x,\lambda^2t),\]
and with the initial data transformation
$u_{0,\lambda}(x)=\lambda u_0(\lambda x)$ and $\theta_\lambda(x)=\lambda^3\theta_0(\lambda x)$.
Notice that,
\[
\|u_{0,\lambda}\|_3=\|u_0\|_3\qquad\text{and}\qquad\|\theta_{0,\lambda}\|_1=\|\theta_0\|_1,
\]
where $\|\cdot\|_p$ denotes the $L^p$-norm,
and these scaling relations motivate the choice of the space
\begin{equation}
\label{uniclass}
C([0,T],L^3(\R^3))\times C([0,T], L^1(\R^3))
\end{equation}
for solving the Boussinesq equations.
The unboundedness of the bilinear operator $B_1$ in $C([0,T],L(\R^3))$ leads to construct solutions
of~\eqref{IEE} applying the usual fixed point not directly in the space~\eqref{uniclass}, but in a Kato's-type smaller space, respecting the same scaling properties as in~\eqref{uniclass}.
For this reason, let us denote by $X$  the subspace of $C([0,T],L^3(\R^3))$, normed by
\begin{align}
 & \|u\|_{X}\equiv \sup_{t\in[0,T]}\|u(t)\|_3+\sup_{0<t\le T}\sqrt t\|u(t)\|_\infty,
\end{align}
and consisting of all divergence-free
vector fields in $C([0,T],L^3(\R^3))$ such that $ \|u\|_X<\infty$ and
$\lim_{t\to0}\sqrt t\|u(t)\|_\infty=0$.
Similarly, let us denote by $Y$ the subspace of  $C([0,T],L^1(\R^3))$, normed by
\begin{align}
 & \|\theta\|_{Y}=\sup_{t\in[0,T]}\|\theta(t)\|_1 + \sup_{0<t\le T} t^{3/2} \|\theta(t)\|_\infty,
 \end{align}
 and consisting of all functions $\|\theta\|_Y<\infty$ and 
 $\lim_{t\to0}t^{3/2}\|\theta(t)\|_\infty=0$.
Then, when $u_0\in L^3(\R^3)$ is divergence-free and $\theta_0\in L^1(\R^3)$, it is easy to establish, just by suitably adapting classical Kato's method~\cite{Kat84} for the Navier--Stokes equations, the following basic existence and uniqueness result in the space $X\times Y$:

\begin{proposition}
\label{th:exi}
Let $u_0\in L^3(\R^3)$ be a divergence-free vector field and 
let $\theta_0\in L^1(\R^3)$.
Then there exists $T>0$ and a unique mild solution $(u,\theta)\in X\times Y$ of~\eqref{IE}.
\end{proposition}
The above solution is global-in-time when, e.g., $\|u_0\|_3+\|\theta_0\|_1$ is small enough.
We refer in this case to \cites{BraM17, BraS12} for the study of their long time behavior, which strikingly differs from the usual behavior as $t\to+\infty$ of solutions of the Navier--Stokes equations.

One can establish several variants of Proposition~\ref{th:exi}. 
For example, we will state a much more general local existence result in Theorem~\ref{theo:exi},
where the $L^3(\R^3)$ and $L^1(\R^3)$ will be replaced by considerably larger Besov spaces.

The main drawback of Proposition~\ref{th:exi}
is that \emph{the uniqueness of the solution is not ensured
in the natural class~\eqref{uniclass}}, but only in the considerably smaller class $X\times Y$.
In this sense, the uniqueness result of the above theorem looks far from being optimal.

In fact, in the case of the Navier--Stokes equations, {\it i.e.\/} when $\theta\equiv0$, Kato's existence result of solutions in $C([0,T],L^3(\R^3))$ is completed by
the well-known uniqueness theorem of Furioli, Lemari\'e-Rieusset, Terraneo~\cite{FurLT00}, stating that, for $u_0\in L^3(\R^3)$,
there is \emph{only one  mild solution} of the Navier--Stokes equations in $C([0,T],L^3(\R^3))$, such that $u(0)=u_0$. See also \cites{Mey97, Mon} for simpler proofs of this important result.
Unfortunately, in the case of the Boussinesq system, it seems difficult
to establish the uniqueness of mild solutions in the natural class~\eqref{uniclass}. 
Indeed, no specific regularity result is available for solutions in such class: if we put no additional condition on the regularity of $u$ or $\theta$ then
the term $\theta u$ appearing in the equation of the temperature is not even a distribution, so that
giving a sense to the term $B_3(u,\theta)$ would be problematic. 

The purpose of this paper is to put in evidence alternative uniqueness classes for the
solutions of the Boussinesq equations.
In this direction, our first main uniqueness result is the following theorem:

\begin{theorem}[Uniqueness]
\label{th:uni2}
Let $T>0$, $u_0\in L^3(\R^3)$ and $\theta_0\in L^1(\R^3)$, with $\nabla\cdot u_0=0$.
Let $(u,\theta)$, $(\tilde u,\tilde \theta)$ be two mild solutions of the Boussinesq system~\eqref{IE} with the same data $(u_0,\theta_0)$, such that
\begin{equation}
 \label{unispa}
 u,\tilde u\in C([0,T],L^3(\R^3)),\qquad\text{and}\qquad
 \theta,\tilde\theta\in  C([0,T],L^1(\R^3))\cap L^\infty_{\rm loc}((0,T),L^{q,\infty}(\R^3)).
\end{equation}
for some $q>3/2$.
Then, $(u,\theta)=(\tilde u,\tilde \theta)$.
\end{theorem}

Theorem~\ref{th:uni2} ensures that the uniqueness holds in a space considerably larger than $X\times Y$. In particular, the vanishing of the $L^{q,\infty}$-norm of $\theta(t)$ as $t\to0$ is not required for the uniqueness.

Let us recall that, if $\sigma>0$ and $1\le q\le\infty$, then a
tempered distribution $f$ satisfies
\begin{equation}
\label{besoc}
\sup_{0<t<T} t^{\sigma/2} \|e^{t\Delta} f\|_q<\infty
\end{equation}
for all $0<T<\infty$ if and only if $f\in B^{-\sigma}_{q,\infty}(\R^3)$. For different values of $T$, all these
expressions are equivalent to the usual inhomogeneous Besov norm $\|\cdot\|_{B^{-\sigma}_{q,\infty}}$.
If \eqref{besoc} holds with $T=\infty$, then $f$ belongs  to the smaller homogeneous
Besov space $\dot B^{-\sigma}_{q,\infty}(\R^3)$ and the converse is also true.
See \cite{Lem02}.
By analogy, we define $B^{-\sigma}_{q,\infty,\infty}(\R^3)$ as the space of tempered distributions $f$ such that,
for some $T>0$, $\sup_{0<t<T} t^{\sigma/2}\|e^{t\Delta}f\|_{L^{q,\infty}}<\infty$.
Before stating our next theorem, let us observe that the 
the solution obtained in Proposition~\ref{th:exi}
satisfies, for all $1<q<\infty$, by interpolation,
\begin{align}
\label{normY}
\sup_{0<t<T}t^{\frac32(1-\frac1q)}\|\theta(t)\|_{L^{q,\infty}}<\infty,\qquad\text{and}\qquad
\lim_{t\to0}t^{\frac32(1-\frac1q)}\|\theta(t)\|_{L^{q,\infty}}=0.
\end{align}
In the sequel, we will denote by $Y_{q,\infty}$ the subspace of $L^\infty_{\rm loc}((0,T),L^{q,\infty}(\R^3))$
made of functions satisfying~\eqref{normY}.
This space is equipped with the natural norm
\[
\|\theta\|_{Y_{q,\infty}}\equiv\sup_{0<t<T}t^{\frac32(1-\frac1q)}\|\theta(t)\|_{L^{q,\infty}}.
\]

When restricting to temperatures in $Y_{q,\infty}$ the uniqueness can be granted as soon as the velocity is in $C([0,T],L^3(\R^3))$, in this case it is no longer needed  to require that $\theta$ belongs to $C([0,T],L^1(\R^3))$.
More precisely we have the following variant of our uniqueness result:

\begin{theorem}
\label{th:uni}
Let $T>0$ and $(u,\theta)$ a mild solution of the Boussinesq system~\eqref{IE}, such that
\begin{equation}
 \label{unispa2}
 (u,\theta)\in C\bigl([0,T],L^3(\R^3)\bigr)\,\times\,\,Y_{q,\infty},
\end{equation}
for some $3/2<q<3$.
Then the data $(u_0,\theta_0)$ belong to $L^3(\R^3)\times  B^{-3(1-1/q)}_{q,\infty,\infty}$
and  uniquely determine $(u,\theta)$.
\end{theorem}

In Section~\ref{sec:proofs} we prove Theorem~\ref{th:uni}, after establishing the relevant bilinear estimates in Lorentz spaces, extending those of~Y.~Meyer in~\cite{Mey97}.
The proof of Theorem~\ref{th:uni2} will rely on the result of Theorem~\ref{th:uni}. Next step consists in 
establishing some fine existence results of solutions, encompassing Proposition~\ref{th:exi},
in the same spirit as Cannone's~\cite{Can04}.
The last step of the proof of Theorem~\ref{th:uni2} consists in removing the restriction $\theta\in Y_{q,\infty}$: this will be done by proving that any mild solution in the class
$C([0,T],L^3(\R^3))\times C([0,T],L^1(\R^3))\cap L^\infty_{\rm loc}((0,T),L^{q,\infty}(\R^3))$,
must agree with the solution of Proposition~\ref{th:exi}. This last step makes use of a compactness argument inspired an earlier uniqueness theorem by H.Brezis on the vorticity equation~\cite{Bre94}.

Our estimates break down in the case $q=3/2$. For for this reason, we do not know, for example, if $C([0,T],L^3(\R^3))\times C([0,T],L^{3/2}(\R^3))$, or $C([0,T],L^3(\R^3))\times Y_{3/2, \infty}$
are uniqueness classes for mild solutions of the Boussinesq system.

On the other hand, in our uniqueness results, it is possible to relax a little bit the $C([0,T],L^3(\R^3))$-condition on the velocity, and to replace it by a weaker condition of the form
$u\in C([0,T],D)$, where $D$ is the closure in 
$L^{3,\infty}(\R^3)$ of $\{f\in L^{3,\infty}\colon -\Delta f\in L^{3,\infty}\}$.
See Remark~\ref{improl} below. In the case of the Navier--Stokes equations,
an even finer uniqueness result is contained in the recent preprint by T.Okabe and Y.Tsutsui~\cite{OkaT}.

While there exists a rich literature on the uniqueness of solutions of the Navier--Stokes equations,
(see, e.g., \cites{Che99, FurLT00, LioM, FarNT, OkaT} for a small sample of the available results), only few earlier papers dealt with the uniqueness problem for the Boussinesq equations
in scale-invariant spaces. Moreover, such papers study, in fact, more or less different versions of the original  system~\eqref{B} (a system with no diffusivity for the temperature in \cite{DanP08}, or a nonlinear diffusivity in \cite{Abi09}, etc.), so that the uniqueness results therein are not comparable to ours.

\section{Proof of the main theorems}
\label{sec:proofs}

\subsection{Preliminary estimates}
\label{sec:esti}

To establish Theorem~\ref{th:uni}, inspired by~\cites{KozY95, Mey97}, we will make use of the Banach space 
$X_{3,\infty}=L^\infty((0,T),L^{3,\infty})$, normed by
\[
\|u\|_{X_{3,\infty}}=\sup_{t\in(0,T)}\|u(t)\|_{L^{3,\infty}(\R^3)}.
\]
We will make use also of the space $X_p$, consisting of the subspace of~$L^1_{\rm loc}((0,T),L^p)$ made of
the vector fields $u$ such that $\|u\|_{X_p}<\infty$, where the $X_p$-norm is defined as
\[
\|u\|_{X_p} =\sup_{t\in(0,T)}t^{\frac{1}{2}(1-3/p)}\|u(t)\|_p, \qquad1\le p\le\infty.
\]
Of special importance will be the space the case $p=3$: in this case $X_3=L^\infty((0,T),L^3(R^3))$.
Concerning the temperature, we will often work in the space $Y_q$ of all the $L^1_{\rm loc}((0,T),L^q(\R^3))$
functions such that $\|\theta\|_{Y_q}<\infty$, where
\[
\|\theta\|_{Y_q} =\sup_{t\in(0,T)}t^{\frac{3}{2}(1-1/q)}\|\theta(t)\|_q, \qquad1\le q\le\infty.
\]
Notice that $Y_1=L^\infty((0,T),L^1(\R^3))$.

The kernel $K(x,t)$ of the operator $e^{t\Delta}\P$ satisfies
\begin{equation}
 \label{kerK}
 K(x,t)=t^{-3/2}K(\textstyle\frac{x}{\sqrt t},1), \qquad\hbox{and}\qquad 
|K(x,1)|\le C(1+|x|)^{-3}.
\end{equation}
In particular, $ K(\cdot,1)\in \bigcap_{1<p\le\infty} L^p(\R^3)$ and
\begin{equation}
 \label{lpK}
 \|e^{t\Delta}\P\theta_0e_3\|_s=Ct^{-\frac32(\frac1r-\frac1s)}\|\theta_0\|_r, \qquad 1\le r< s\le \infty.
\end{equation}
On the other hand, the kernel $F(x,t)$ of the operator $e^{t\Delta}\P\div$ satisfies
\begin{equation}
 \label{kerF}
 F(x,t)=t^{-2}F(\textstyle\frac{x}{\sqrt t},1), \qquad\hbox{and}\qquad
 |F(x,1)|\le C(1+|x|)^{-4}.
\end{equation}
See \cite{Miy00}.
In particular, $F(\cdot,1)\in L^1\cap L^\infty$ and
\begin{equation}
 \label{lpF}
 \|F(t)\|_\beta=  C t^{-2+3/(2\beta)}, \qquad 1\le \beta\le \infty.
\end{equation}

The following estimates are well known to Navier--Stokes specialists, see \cite{Mey97}*{Lemma~23}. Only the
first one is subtle, the second one being just an application of H\"older and Young inequality in Lorentz spaces:
\begin{subequations}
\begin{align}
\label{est:b1w}
  \|B_1(u,v)\|_{X_{3,\infty}} &\le C\|u\|_{X_{3,\infty}}\|v\|_{X_{3,\infty}},\\
  \label{est:b1wbis}
  \|B_1(u,v)\|_{X_{3,\infty}} &\le C\|u\|_{X_{3,\infty}}\|v\|_{X_\infty}.
\end{align}

The counterpart of~\eqref{est:b1w} for the operator~$B_3$ is stated below.
With slightly abusive notation we will denote in the same way, by $F(x,t)$
the kernel of the operators $e^{t\Delta}\P\div$ and $e^{t\Delta}\div$.
Distinguishing these two kernels is unimportant in this paper because both
kernels satisfy properties~\eqref{kerF} and \eqref{lpF}, that are the only properties that we will need.

\begin{lemma}
Let $3/2< q<3$. If $u\in X_{3,\infty}$ and $\theta\in Y_{q,\infty}$, then $B_3(u,\theta)\in Y_{q}$.
Moreover, there exists a constant $C>0$,depending only on~$q$, such that, for all $u$ and $\theta$,
\label{lem:b3}
\begin{equation}
\label{est:b3}
  \|B_3(u,\theta)\|_{Y_{q,\infty}}\le C\|u\|_{X_{3,\infty}}\|\theta\|_{Y_{q,\infty}}.
\end{equation}
 \end{lemma}

\begin{proof}
Let us recall that the quasi-norm
\[
 f\mapsto  \sup_{\lambda>0} \lambda\Bigl|\{x\in\R^3\colon |f(x)|>\lambda\}\Bigr|^{1/q}.
\]
is equivalent to a norm that makes $L^{q,\infty}$ a Banach space. Here $|A|$ denotes the Lebesgue measure of the
set~$A$. With slightly abusive notation
we denote $\|f\|_{L^{q,\infty}}$ the right-hand side of the above expression, and treat it as a norm. 
Let us set $\sigma=\frac32(1-1/q)$.
Without loss of generality, we can assume that
$\|u\|_{X_{3,\infty}}=1$ and $\|\theta\|_{Y_{q,\infty}}=1$. In the computations below, $C$ will denote absolute constants.
We start splitting
\[
t^\sigma B_3(u,\theta)= -t^\sigma\biggl(\int_0^{t/2}\;+\;\int_{t/2}^t\biggr) F(t-s)*(u\theta)(s)\,ds
\equiv {\rm (I)+(II)}.
\]
To treat ${\rm (I)}$ we only need to apply the standard convolution and H\"older inequality in Lorentz spaces, see~\cite{Lem02}. Using \eqref{lpF} with $\beta=3/2$ we obtain
\[
\| \,{\rm (I)}\,\|_{L^{q,\infty}}\le Ct^\sigma \int_0^{t/2} (t-s)^{-1} s^{-\frac32(1-1/q)}\dd s \le C.
\]
The estimate for ${\rm (II)}$ is less immediate.
Fix a threshold  $\lambda>0$ and let $\tau>0$ to be chosen later.
We now write
\[
\begin{split}
{\rm (II)}&=-t^\sigma \int_{-\infty}^{t-\tau}F(t-s)*(u\theta1_{[\frac t2,t]})(s)\dd s 
 - t^\sigma \int_{t-\tau}^tF(t-s)*(u\theta1_{[\frac t2,t]})(s)\dd s\\
    &\equiv J_1(t)+J_2(t).
 \end{split}
\]
We estimate $\|J_1(t)\|_\infty$ by applying~\eqref{lpF} with $\frac1\beta=\frac23-\frac 1q$.
We obtain
\[
\begin{split}
 \|J_1(t)\|_\infty
 &\le C\int_{-\infty}^{t-\tau} (t-s)^{-1-3/(2q)}\dd s\\
 &=C\tau^{-3/(2q)}
  =\lambda/2.
\end{split}
\]
The choice of $\tau>0$ is made in order to ensure the validity of the last equality.

The two relations above imply that $|{\rm (II)}(x,t)|\le \frac \lambda2  + |J_2(x,t)|$. 
Hence, for $\frac 1r=\frac13+\frac1q$, 
\[
\
\begin{split}
 \Bigl|\{x\in\R^3\colon | {\rm (II)}(x,t)|>\lambda\}\Bigr| 
 &\le \Bigl|\{x\in\R^3\colon |J_2(x)|>\lambda/2\}\Bigr|\\
 &\le\biggl(\frac{2\|J_2(t)\|_{L^{r,\infty}}}{\lambda}\biggr)^{r},
\end{split}
\]
where the second inequality follows from the definition of $\|\cdot\|_{L^{r,\infty}}$.
On the other hand,
\[
\begin{split}
 \|J_2(t)\|_{L^{r,\infty}}
 &\le t^\sigma \int_{t-\tau}^t \|F(t-s)\|_1 \|(u\theta1_{[\frac t2,t]})(s)\|_{L^{r,\infty}}\dd s\\
 &\le C\tau^{1/2}=C\lambda^{-q/3}.
\end{split}
\]
From the last two inequalities we deduce that
\[
\lambda\Bigl|\{x\in\R^3\colon | {\rm (II)}(x,t)|>\lambda\}\Bigr|^{1/q} \le C,
\]
where $C$ is independent on~$\lambda$.
This gives estimate~\eqref{est:b3}.
\end{proof}

\begin{lemma}
 \label{lem:b123}
 Let $1\le p_0,p_1,p_2\le \infty$, and $1\le q_0,q_1\le\infty$.
  Then the following estimates hold,
 for some constant $C>0$ depending only on the above parameters (in particular, $C$ is independent on~$T$):
\begin{align}
\label{est:b1}
\|B_1(u,v)\|_{X_{p_0}}
&\le C\|u\|_{X_{p_1}}\|v\|_{X_{p_2}}
 \qquad 
 \textstyle
 (\frac1{p_0}\le \frac1{p_1}+\frac1{p_2}\le 1, \quad 0< \frac1{p_1}+\frac1{p_2}<\frac13+\frac1{p_0}),\\
\label{est:b2}
\|B_2(u,\theta)\|_{X_{p_0}}
&\le C\|u\|_{X_{p_1}}\|\theta\|_{Y_{q_1}}\;
 \qquad 
 \textstyle
 (\frac1{p_0}\le \frac1{p_1}+\frac1{q_1}\le 1, \quad \frac23< \frac1{p_1}+\frac1{q_1}<1+\frac1{p_0}),\\
\|B_3(u,\theta)\|_{Y_{q_0}}
&\le C\|u\|_{X_{p_1}}\|\theta\|_{Y_{q_1}}\,
 \qquad 
 \textstyle
 (\frac1{q_0}\le \frac1{p_1}+\frac1{q_1}\le 1, \quad \frac23< \frac1{p_1}+\frac1{q_1}<\frac13+\frac1{q_0}).
\end{align}  
\end{lemma}

\begin{proof}
The proof just consists in applying Young and H\"older inequality. For~\eqref{est:b1}, one uses~\eqref{lpF} with 
$1+\frac1\beta=\frac{1}{p_0}+(\frac1{p_1}+\frac1{p_2})$ (obvious modification of the choice of
$\beta$ for the two other estimates). By the definition of the $X_p$ and the $Y_q$ norms, one
ends up with integrals of the form $\int_0^t(t-s)^{\alpha}s^\beta\dd s$, that are all finite
because our restrictions on the parameters imply $\alpha,\beta>-1$. One concludes observing that
these integrals are equal to $Ct^{\alpha+\beta+1}$, with $C>0$ independent on~$t$.
Let us mention that estimate~\eqref{est:b1} already appears in the Navier--Stokes literature, see~\cite{BaCD}.
\end{proof}

Let us observe that in the last estimate of Lemma~\ref{lem:b123} the limit 
case $q_0=q_1$ and $p_1=3$ is forbidden.
Lemma~\ref{lem:b3}, however,  provides a substitute of this estimate for corresponding
the weak norms.
In the same way, the first estimate of Lemma~\ref{lem:b123} in the limit case $p_0=p_1=p_2=3$ is forbidden.
But this estimate has a substitute  for the  corresponding weak norms, given by~\eqref{est:b1w}.

We will make use of weak variants of estimates~\eqref{est:b2}, namely
\begin{align}
 \label{est:b2w}
\|B_2(u,\theta)\|_{X_{3,\infty}}
&\le C\|u\|_{X_{3,\infty}}\|\theta\|_{Y_{q,\infty}}\qquad
(\textstyle\frac13< \frac1{q}<\frac23),\\
\label{est:b2ww}
\|B_2(u,\theta)\|_{X_{3,\infty}}
&\le C\|u\|_{X_{p_1}}\|\theta\|_{Y_{q,\infty}}\qquad
(\textstyle\frac23\le \frac1{p_1}   +\frac1q\le 1),\\
\label{est:b3w}
\|B_3(u,\theta)\|_{Y_{q,\infty}}
&\le C\|u\|_{X_{p}}\|\theta\|_{Y_{q,\infty}}\,
 \qquad 
 \textstyle
 (\frac1p+\frac1q< 1, \quad \frac23< \frac1p+\frac1q<\frac13+\frac1q).
\end{align}
\end{subequations}
The proof is essentially the same as in Lemma~\ref{lem:b123}. 
Notice that the case $\frac1{q}=\frac23$ in the former estimate and the case $\frac 1p+\frac 1q=1$ in the latter
have to be excluded.

%%%%%%
%{\color{gray}

%\begin{lemma}
%\label{lem:b12}
%There exists $C>0$, depending only on $q$  such that for all $u$, $v$ and $\theta$
%\begin{align}
%%\label{est:b1}
% % \|B_1(u,v)\|_{X_{3,\infty}} &\le C\|u\|_{X_{3,\infty}}\|v\|_{X_{3,\infty}},\\
% \label{est:b1bis}
%  \|B_1(u,v)\|_{X_{3,\infty}} &\le C\|u\|_{X_{3,\infty}}\|v\|_{X_\infty},
%  \qquad 
%  \|B_1(u,v)\|_{X_3} \le C\|u\|_{X_3}\|v\|_{X_\infty},
%\end{align}
%and, for $3/2< q<3$, 
%\begin{align}
%  \qquad
%  \|B_2(u,\theta)\|_{X_3}\le C\|u\|_{X_3}\|\theta\|_{Y_{q,\infty}},\\
%  \label{est:b2bis}
%   \|B_2(u,\theta)\|_{X_\infty} &\le C\|u\|_{X_\infty}\|\theta\|_{Y_{q,\infty}}.
%\end{align}
%\end{lemma}

%
%{\tt !!!!!! probably here more general $X_p$- estimates are needed. -------}

%\begin{proof}
%Inequalities~\eqref{est:b1}-\eqref{est:b1ter} are well known to Navier--Stokes specialists, see, e.g., \cite{Mey97}.
%The two estimates~\eqref{est:b2} and estimate~\eqref{est:b2bis} are elementary:
%they immediately follow applying Young and H\"older-type estimates
%and equation~\eqref{lpF} with $\beta=q'$ (the conjugate exponent of $q$).
%\end{proof}

%Contrary to~\eqref{est:b1bis} and~\eqref{est:b2}, in estimate~\eqref{est:b1} the use of the $X_{3,\infty}$ norm is really essential, as the same kind of estimate is known to be wrong for the $X_3$-norm. 
%This nice observation was made in the unpublished Ph.D thesis of  F. Oru (2001).

%For the same reason, in Lemma~\ref{lem:b3} the use of the weak spaces, instead of the usual Lebesque spaces,
%seems to be unavoidable.
%} 

\subsection{The proof of Theorem~\ref{th:uni}.}

We are now in the position of establishing Theorem~\ref{th:uni}.
 
\begin{proof}[Proof of Theorem~\ref{th:uni}]
Consider a solution $(u,\theta)$ satisfying the conditions of Theorem~\ref{th:uni}.
We have of course $u_0\in L^3(\R^3)$. 
By the equation satisfied by~$\theta$ in~\eqref{IEE} and 
Lemma~\ref{lem:b3} we have $\|e^{t\Delta}\theta_0\|_{Y_{q,\infty}}\le \|\theta\|_{Y_{q,\infty}}+C\|u\|_{X_{3,\infty}}\|\theta\|_{Y_{q,\infty}}$.
Recalling the observation right after~\eqref{besoc}, we see that $\theta_0\in B^{-3(1-1/q)}_{q,\infty,\infty}$.

Now let $(\tilde u,\tilde\theta)$ be another solution in $C([0,T],L^3(\R^3))\times Y_{q,\infty}$ arising from the same data.
Let $w=u-\tilde u$ and $\phi=\theta-\tilde\theta$.
Then,
\[
\begin{split}
 \label{eq:wphi}
    w&=B_1(w,u)+B_1(\tilde u,w)+B_2(w,\theta)+B_2(\tilde u,\phi),\\
 \phi&=B_3(u,\phi)+B_3(w,\tilde\theta).
\end{split}
\]
\begin{subequations}
\label{sube}
Adding/substracting to $u$ and $\tilde u$ the linear quantity $v_0=e^{t\Delta}[u_0+t\P\theta_0e_3]$,
we find by estimates~\eqref{est:b1w}-\eqref{est:b1wbis},
\begin{equation}
\label{absoo}
\begin{split}
\| B_1(w,u)+B_1(\tilde u,w) \|_{X_3,\infty} 
\le C \|w\|_{X_{3,\infty}}\Bigl( \|u-v_0\|_{X_{3,\infty}}+ 2\|v_0\|_{X_\infty}   
   + \|\tilde u-v_0\|_{X_{3,\infty}}.
   \Bigr)
\end{split}
\end{equation}
Recall that, by our assumption, $\frac13<\frac1q<\frac23$. Hence,
we can apply~\eqref{est:b2w}-\eqref{est:b2ww} choosing $p_1$ in a such way that $p_1>3$.
Then we get: 
\begin{equation}
\label{absooo}
\begin{split}
\|B_2(w,\theta)+B_2(\tilde u,\phi)\|_{X_{3,\infty}}
\le C\|w\|_{X_{3,\infty}}\|\theta\|_{Y_{q,\infty}}
+ C\|\phi\|_{Y_{q,\infty}}\Bigl(\|\tilde u-v_0\|_{X_{3,\infty}}+\|v_0\|_{X_{p_1}}\Bigr).
\end{split}
\end{equation}
\end{subequations}

\begin{subequations}
\label{subee}
Combining the two last estimates we get
\begin{equation}
\label{abso}
\begin{split}
 \|w\|_{X_{3,\infty}}
 &\le C\|w\|_{X_{3,\infty}}\Bigl( \|u-v_0\|_{X_{3,\infty}}+ 2\|v_0\|_{X_\infty}   
   + \|\tilde u-v_0\|_{X_{3,\infty}}  +\|\theta\|_{Y_{q,\infty}}
   \Bigr)\\
    &\quad\qquad\qquad\qquad \qquad\qquad\qquad
   + C\|\phi\|_{Y_{q,\infty}}\Bigl(\|\tilde u-v_0\|_{X_{3,\infty}}+\|v_0\|_{X_{p_1}}\Bigr).\end{split}
\end{equation}

We estimate $ \|\phi\|_{Y_{q,\infty}}$ by applying Lemma \ref{lem:b3} and estimate \eqref{est:b3w} with $3/2<q<3$ and $p= q^*$,
\begin{equation}
	\label{abso2}
	\begin{split}
	\|\phi\|_{Y_{q,\infty}} 
	& =  \| B_3(u - v_0 +v_0  ,\phi)\|_{Y_{q,\infty}} + \|B_3(w,\tilde\theta) \|_{Y_{q,\infty}}  \\
	&\le C\|\phi\|_{Y_{q,\infty}}\Bigl(\|u-v_0\|_{X_{3,\infty}}+\|v_0\|_{X_{q^*}}\Bigr) 
	+C\|w\|_{X_{3,\infty}} \|\tilde\theta\|_{Y_{q,\infty}}
	\end{split}
	\end{equation}
where $\frac1{q^*}=\frac12(1-\frac1q)$. This choice of $q^*$ ensures that~\eqref{est:b3w} holds and $3<q^*<\infty$.
\end{subequations}
In~\eqref{subee}, the constants $C>0$ depends only on~$q$.

On the other hand all the norms in~\eqref{abso} depend on $T$. We claim that
\[
\|v_0\|_{X_\infty} \to0, \qquad \text{as $T\to0$}.
\]
This can be seen as follows: first of all by our assumption 
$(u,\theta)\in X_3\times Y_{q,\infty}\subset X_{3,\infty}\times Y_{q,\infty}$, and so
$B_3(u,\theta)\in Y_{q,\infty}$ by~Lemma~\ref{lem:b3}. 
In particular, because of the definition of $Y_{q,\infty}$ (see~\eqref{normY}),
$\|e^{t\Delta}\theta_0\|_{Y_{q,\infty}}\le \|\theta(t)\|_{Y_{q,\infty}}+\|B_3(u,\theta)\|_{Y_{q,\infty}}\to0$ as $T\to0$.
Leray's projector $\P$ being bounded on $L^{q,\infty}$ we deduce $\|e^{t\Delta}\P\theta_0e_3\|_{Y_{q,\infty}}\to0$
as $T\to0$.
Applying the semigroup property $e^{t\Delta}=e^{t\Delta/2}e^{\Delta/2}$ and using the boundedness properties of $e^{t\Delta/2}$ in Lorentz spaces, we deduce  $\|t\,e^{t\Delta}\P\theta_0e_3\|_{X_\infty}\to0$
as $T\to0$.
Moreover, $u_0\in L^3(\R^3)$, hence
for any $\epsilon>0$, we can find $u_{0,\epsilon}$ in the Schwartz class,
such that $\|u_0-u_{0,\epsilon}\|_3<\epsilon$, and 
$\|u_{0,\epsilon}\|_3\le \|u_0\|_3$.
Writing 
\[
v_0=e^{t\Delta}[(u_0-u_{0,\epsilon})]+e^{t\Delta}u_{0,\epsilon}
   +te^{t\Delta}\P\theta_0e_3
\]
we deduce from Young inequality that $\|v_0\|_{X_\infty}\le 2\epsilon$ for $T>0$ small enough
and our claim follows.

With a very similar proof (using $q<3<q^*$ and $p_1>3$) we see that 
\[
\|v_0\|_{X_{q^*}}\to0 \qquad\text{and}\qquad \|v_0\|_{X_{p_1}}\to0, \qquad \text{as $T\to0$},
\]

Next we claim that
\begin{equation}
\label{claa}
\|u-v_0\|_{X_{3,\infty}}+\|\tilde u-v_0\|_{X_{3,\infty}}\to0, \qquad\text{as $T\to0$}.
\end{equation}
Indeed, we use the inequality
\[
\|u-v_0\|_{X_{3,\infty}}+\|\tilde u-v_0\|_{X_{3,\infty}}
\le \|u-e^{t\Delta}u_0\|_{X_3}+\|\tilde u-e^{t\Delta}u_0\|_{X_3}
+2\|te^{t\Delta}\P\theta_0e_3\|_{X_{3,\infty}}.\]
Next, we use the fact that, as $t\to0$,  $u(t)\to u_0$ and $\tilde u(t)\to u_0$ in $L^3$ (because of
the continuity of the two solutions $u$ and $\tilde u$ from $[0,T]$ to $L^3(\R^3)$), and also that $e^{t\Delta}u_0\to u_0$ in $L^3(\R^3)$.
Next we observe that $\|te^{t\Delta}\P\theta_0e_3\|_{X_{3,\infty}}\to0$ as $T\to0$.
(The last fact is proved applying the semigroup properties of the heat kernel and 
the fact that $\|e^{t\Delta}\P\theta_0e_3\|_{Y_{q,\infty}}\to0$ as $T\to0$).
This implies our claim~\eqref{claa}. 

On the other hand, our assumptions on $\theta$ and $\tilde\theta$ ensure that,
\[
\|\theta\|_{Y_{q,\infty}}\to0 \qquad\text{and}\qquad \|\tilde \theta\|_{Y_{q,\infty}}\to0,\qquad\text{as $T\to0$}.
\]

Summarizing, we can now deduce from estimates~\eqref{sube} that there exists
$\delta>0$ (depending only on the data $(u_0,\theta_0)$
and on $q$), such that, if $0<T<\delta$, then $\|w\|_{X_{3,\infty}}<\|\phi\|_{Y_{q,\infty}}$ 
and $\|\phi\|_{Y_{q,\infty}}<\|w\|_{X_{3,\infty}}$.
This implies $w=\phi=0$, and so
$u(t)=\tilde u(t)$, $\theta(t)=\tilde \theta(t)$ for all $t\in[0,T]$.

When $T\ge \delta$, the above argument implies only that $u(t)=\tilde u(t)$ and $\theta(t)=\tilde \theta(t)$ for all 
$t\in[0,\delta)$.
But, if $\tau>0$ is the supremum of $t\in[0,T]$ such that $(u,\theta)$ and $(\tilde u,\tilde\theta)$ agree 
on $[0,t]$, 
then $u(\tau)=\tilde u(\tau)$ and $\theta(\tau)=\tilde\theta(\tau)$ by the time-continuity
assumption on the solutions.
Then $\tau=T$, as otherwise considering the new data at the time $\tau$, we could apply the above uniqueness result in the interval $[\tau,\tau+\delta)$.
This is indeed possible, since it is obvious by Lemma~\ref{lem:b3}
that the solutions remain in the space $X_{3,\infty}\times Y_{q,\infty}$ after the time translation.
We thus would contradict, the definition of~$\tau$.
The assertion of Theorem~\ref{th:uni} follows.
\end{proof}

\begin{remark}
\label{improl}
The above proof shows that the uniqueness for mild solutions of the Boussinesq system holds, in fact, in a class that is larger than $C([0,T],L^3(\R^3))\times Y_{q,\infty}$.
Indeed, let $D$ be the closure in $L^{3,\infty}$ of 
$\{f\in L^{3,\infty}(\R^3)\colon \Delta f\in L^{3,\infty}(\R^3)\}$.
The space $D$ was characterized by Lunardi~\cite{Lun} to be the maximal subspace in $L^{3,\infty}$
where the Stokes semigroup is $C_0$-continuous, i.e., $D$ is the space of all the $L^{3,\infty}$ functions~$f$ such that 
\begin{equation}
\label{lunac}
\lim_{\epsilon\to0+} \|e^{\epsilon\Delta} f-f\|_{L^{3,\infty}}=0.
\end{equation}
See also~\cite{OkaT} for a direct proof of this fact.
If $u_0\in L^{3,\infty}(\R^3)$ is divergence free and satisfy~\eqref{lunac} with $u_0$ instead of $f$, then our proof goes through. We thus obtain the uniqueness in the larger class
$C([0,T],D)\times Y_{q,\infty}$, with $3/2<q<3$.
Notice that $D$ is strictly larger than $L^3(\R^3)$, and it is larger also than the closure of smooth compactly supported functions in $L^{3,\infty}$ as, for example, smooth functions decaying like 
$\sim|x|^{-1}$ at infinity do belong to~$D$.
%For example, an initial velocity like $(-\frac{2x_2x_3^2}{(1+|x|^2)^2},\frac{2x_1x_3^2}{(1+|x|^2)^2},0)\in D$,
%and an initial temperature $\theta_0\in L^1(\R^3)$ uniquely determine the solution in the above class.
\end{remark}

\subsection{Existence theorems.}

We start by establishing a quite general local existence result.

\begin{theorem}
\label{theo:exi}

Let $3/2<q<3$ and $\theta_0$ be in the closure of the Schwartz class in $B^{-3(1-1/q)}_{q,\infty}(\R^3)$.
Let $p>3$ such that $\frac23<\frac1p+\frac1q$ and $u_0$ a divergence-free vector field in the closure of the
the Schwartz class in the inhomogeneous Besov space $B^{-(1-3/p)}_{p,\infty}(\R^3)$.
\begin{itemize}
\label{itemi}
\item[(i)]
Then there exists $T>0$ and a solution $(u,\theta)$ of~\eqref{IE},
such that $(u,\theta)\in X_{p}\times Y_q$ and
$\|u\|_{X_p}\to0$, $\|\theta\|_{Y_q}\to0$ as $T\to0$.
Moreover there exists $R>0$, depending only on~$p$ and $q$,
such that $(u,\theta)$ is the only solution satisfying $\|u\|_{X_p}<R$
 and $\|\theta\|_{Y_q}<R$.
\label{itemii}
\item[(ii)](Regularity)
The above solution belongs in fact to $(X_p\cap X_\infty)\times (Y_q\cap Y_\infty)$.
Moreover,
$\|u\|_{X_\infty}\to0$ and $\|\theta\|_{Y_\infty}\to0$ as $T\to0$.
\label{itemiii}
\item[(iii)]
\begin{itemize}
\item[-] Under the more stringent condition $u_0\in L^3(\R^3)$, we have $u\in X\subset X_3\cap X_\infty$.
\item[-] Under the more stringent condition $\theta_0\in L^1(\R^3)$, we have $\theta \in Y\subset Y_1\cap Y_\infty$.
\end{itemize}
\end{itemize}
\end{theorem}

\begin{proof}[Proof of Theorem~\ref{theo:exi}]
To prove the assertion~(i),
let us write the unknown of the integral Boussinesq equation~\eqref{IEE} as 
${\bf v}=\displaystyle\binom{u}{\theta}$.
Let also
\begin{equation}
\label{defa}
{\bf v_0}=\binom{e^{t\Delta}[u_0+t\P\theta_0e_3]}{e^{t\Delta}\theta_0},
\end{equation}
and
\begin{equation*}
\label{defbb}
{\bf B}({\bf v},{\bf \tilde v})=\binom{B_1(u,\tilde u)+B_2(u,\tilde\theta)}{B_3(u,\tilde\theta)}.
\end{equation*}

Then we see that an equivalent way of writing~\eqref{IEE} is 
\begin{equation}
\label{absy}
{\bf v}= {\bf v_0}+\bf {B}({\bf v},{\bf v}),
\end{equation}
complemented with $\div\, u_0=0$.
We will apply the standard fixed point Lemma~\cite{Mey97}*{Lemma~20} to this equation in the Banach
space $E=X_p\times Y_q$. 
To achieve this, we only need to prove the existence of a constant $C_0>0$
such that the following estimate holds for all ${\bf v}$ and ${\bf \tilde v}$ in~$E$:
\begin{equation}
 \label{applicab}
\|{\bf B(v,\tilde v)}\|_E\le C_0\|{\bf v}\|_E\|{\bf \tilde v}\|_E,
\end{equation}
and such that
\begin{equation}
\label{data-a}
  \|={\bf v_0}\|_E<1/(4C_0).
\end{equation}
If that is the case, then the fixed point lemma provides the existence of a solution~${\bf v}$ to
the abstract equation~\eqref{absy}, such that 
\begin{equation}
\label{poin}
\|{\bf v}\|_E\le 2\|{\bf v_0}\|_E<\frac1{2C_0}.
\end{equation}
The uniqueness of this solution is a priori ensured only under the condition $\|{\bf v}\|_E<\frac{1}{2C_0}$.

In fact, we already established estimate~\eqref{applicab}: it is an immediate consequence of Lemma~\ref{lem:b123},
applied with $p_0=p_1=p_2=p$ and $q_0=q_1=q$.
The constant $C_0$ can be taken independent on~$T$.
On the other hand, by the characterization of the non-homogeneous Besov space~\eqref{besoc},
the condition $(u_0,\theta_0)\in B^{-(1-3/p)}_{p,\infty}\times B^{-3(1-1/q)}_{q,\infty}(\R^3)$
precisely means ${\bf a}\in E$. But in fact the data belong to closure of the Schwartz class in their respective
spaces, thus by the usual approximation argument we have $\|{\bf v_0}\|_E\to0$ as $T\to0$.
In conclusion, \eqref{data-a} holds true if $T>0$ is small enough.
Recalling~\eqref{poin}, we get the first item of Theorem~\ref{theo:exi}, with $R=1/(2C_0)$.

To prove the assertion~(ii), let us apply the first two estimates of Lemma~\ref{lem:b123} with $p_1=p_2=p$, $q_1=q$ and $\frac2p-\frac13<\frac1{p_0}<\frac1p$.
We obtain in this way $B_1(u,u)\in X_{p_0}$ and $B_2(u,\theta)\in X_{p_0}$.
On the other hand, we know that $e^{t\Delta}v_0\in X_p$, but as $p_0>p$ we have also
$e^{t\Delta} v_0\in X_{p_0}$, by the usual $L^p$-$L^{p_0}$ heat estimates.
Summing these three terms we find $u\in X_{p_0}$. Observe that if $p>6$, then we can directly take $p_0=\infty$,
otherwise we proceed by bootstrapping. After finitely many iterations we find $u\in X_\infty$.

Let us now apply the third estimate of Lemma~\ref{lem:b123} with $q_1=q$, and $p_1=p$, and
$\frac1p+\frac1q-\frac13<\frac 1{q_0}<\frac 1q$.
Then $B_3(u,\theta)\in Y_{q_0}$. Moreover, since $q_0>q$, $e^{t\Delta}\theta_0\in Y_{q_0}$ by
the usual $L^{q}-L^{q_0}$ heat kernel estimates. This implies $\theta\in Y_{q_0}$.
If $\frac1p+\frac 1q-\frac13<0$ then we can take directly $q_0=\infty$. 
Otherwise we proceed by bootstrapping
and after finitely many steps we find $\theta\in Y_\infty$. 

Recalling that $\|u\|_{X_p}\to0$ and $\|\theta\|_{Y_q}\to0$,
as $T\to0$ the above applications of Lemma~\ref{lem:b123} show that $\|u\|_{X_{p_0}}$ 
 and $\|\theta\|_{Y_{q_0}}$ also go to zero with $T$. So at the end of the above
bootstrapping procedures we find $\|u\|_{X_\infty}\to0$ and $\|\theta\|_{Y_\infty} \to0$, as $T\to0$.

Let us now prove the assertion~(iii). 
If $u_0\in L^3(\R^3)$, then for all $3<p\le \infty$ we do have $u_0\in B^{-(1-3/p)}_{p,\infty}(\R^3)$.
We can apply the first two estimates of Lemma~\ref{lem:b123} with $p_0=3$, $p_1=p_2=p$, $q_1=q$,
provided $\frac13\le \frac 2p<\frac 23$ and $\frac23<\frac1p+\frac 1q<\frac43$. This is indeed possible choosing
$p>3$ close enough to~$3$. This shows that both $B_1(u,u)$ and $B_2(u,\theta)$ belong to $X_3$.
Moreover, $e^{t\Delta}u_0\in X_3$ by the usual heat kernel estimates.
We conclude that if $u_0\in L^3(\R^3)$ then the solutions constructed before belongs to $X_3$.

If we now assume  $\theta_0\in L^1(\R^3)$, then $\theta_0\in B^{-3(1-1/q)}_{q,\infty}(\R^3)$
for all $1<q\le \infty$. Hence, by what we already proved, 
we have $(u,\theta)\in (X_3\cap X_\infty)\times (Y_q\cap Y_\infty)$, for all $q>3/2$. In particular,
$(u,\theta)\in X_6\times Y_{12/7}$. 
Let us apply the last estimate of 
Lemma~\ref{lem:b123} with $p_1=6$, $q_1=12/7$ and $q_0=4/3$: then we get
$B_3(u,\theta)\in Y_{4/3}$.
But $\theta_0\in L^1(\R^3)$ implies also $e^{t\Delta}\theta_0\in Y_1\cap Y_\infty\subset Y_{4/3}$.
Therefore, $(u,\theta)\in X_4\times Y_{4/3}$. We now apply the last estimate of Lemma~\ref{lem:b123} with
$p_1=4$, $q_1=4/3$ and $q_0=1$: we get in this way $B_3(u,\theta)\in Y_1$. As we already know that
$e^{t\Delta}\theta_0\in Y_1$, we conclude that $\theta\in Y_1$.

Assume now that we have both conditions  $u_0\in L^3(\R^3)$ and $\theta_0\in L^1(\R^3)$.
According to the definition of $X$ and $Y$, it only remain
to prove that the maps $t\mapsto u(t)$ and $t\mapsto \theta(t)$ are 
continuous form $[0,T]$ respectively to $L^3(\R^3)$ and $L^1(\R^3)$.
This is quite standard. Let us sketch the proof for the time continuity of temperature.
We have of course $e^{t\Delta}\theta_0\in C([0,T],L^1(\R^3))$. Moreover,
recalling that $\|u\|_{X_\infty}\to0$ as $T\to0$,
and applying the third inequality of Lemma~\ref{lem:b123} with $q_0=q_1=1$ and $p_1=\infty$
we see that $\lim_{t\to0} \|B_3(u,\theta)(t)\|_1\to0$. Hence $t\mapsto \theta(t)$ is continuous at
$t=0$ in $L^1(\R^3)$.
If $0<t\leq T$, we consider the expression $\|B_3(u,\theta)(t+h)-B_3(u,\theta)(t)\|_1$.
This can be bounded by the sum of $\|u\|_{X\infty}\|\theta\|_{Y_1}\int_0^t\|F(t+h-s)-F(t-s)\|_1 s^{-1/2}\dd s$
and $\|u\|_{X\infty}\|\theta\|_{Y_1}\int_t^{t+h}\|F(t+h-s)\|_1 s^{-1/2}\dd s$.
Both terms are easily proved to converge to~$0$ as $h\to0$, using properties~\eqref{kerF}-\eqref{lpF}
of the kernel $F$. This establishes the continuity of $\theta$ from $[0,T]$ to $L^1(\R^3)$.
The continuity of $u$ from $[0,T]$ to $L^3(\R^3)$ is treated in the same way.
\end{proof}

\begin{proof}[Proof to Proposition~\ref{th:exi}]
This is now immediate.
If $u_0$ is divergence-free and $(u_0,\theta_0)\in L^3(\R^3)\times L^1(\R^3)$,
then the existence of a local-in time solution $(u,\theta)\in X\times Y$ is already established
in Theorem~\ref{theo:exi}.
If $(\tilde u,\tilde\theta)$ is another solution in $X\times Y$ starting from the same data,
then there is $\delta>0$ such that 
$(\tilde u,\tilde \theta)$ and $(u,\theta)$ agree on $[0,\delta]$. Indeed, 
using the definition of $X$ and $Y$ we find by interpolation
$\lim_{t\to0} t^{\frac12(1-3/p)}\|u(t)\|_p=0$ for all $p>3$ and 
$\lim_{t\to0} t^{\frac32(1-3/q)}\|\theta(t)\|_q=0$ for all $q>1$. Moreover,
the same property holds for $(\tilde u,\tilde\theta)$.
Then, multiplying the two solutions
by the indicator function of the interval $[0,\delta]$,
we see that $\|(u,\theta)1_{[0,\delta]}\|_{E}<R$  and 
$\|(\tilde u,\tilde \theta)1_{[0,\delta]}\|_{E}<R$.
The first statement of~Theorem~\ref{theo:exi} then implies that $(u,\theta)1_{[0,\delta]}$ 
and $(\tilde u,\tilde \theta)1_{[0,\delta]}$ agree.
It is now obvious that the supremum on $[0,T]$ of the times $t$, such that 
the two solutions agree on $[0,t]$, must be equal to $T$. 
In conclusion, there is only one solution in $X\times Y$ arising from $(u_0,\theta_0)$. 
\end{proof}

Proposition~\ref{th:exi} allow us to define a semigroup $R(t)\colon L^3(\R^3)\to L^3(\R^3)$ and a semigroup
$S(t)\colon L^1(\R^3)\to L^1(\R^3)$ such that $(R(t)u_0,S(t)\theta_0)$ is the unique
solution in $X\times Y$ of the Boussinesq system~\eqref{IE}.

\begin{remark}
 \label{rem:time}
The present useful remark is inspired from Ben Artzi's paper~\cite{BenA94}.
In Proposition~\ref{th:exi}, the existence time $T>0$ a priori 
does not only depend on $\|u_0\|_3$ and $\|\theta_0\|_1$,
but also on $u_0$ and $\theta_0$ themselves. However, if we restrict to a class of data 
$(u_0,\theta_0)\in H\times K$, where $H$ is precompact in $L^3(\R^3)$ and $K$ is precompact in $L^1(\R^3)$,
then the existence time $T$ depends only on $H$ and $K$.
Indeed, an elementary property of the heat equation is that, as $t\to0$,
\begin{equation}
\label{eq:uniflim}
\sup_{\theta_0\in K}\Bigl( t^{\frac32(1-1/q)}\|e^{t\Delta}\theta_0\|_q \Bigr)\to0,
\end{equation}
for $1<q\le\infty$.
To see this, fix $\epsilon>0$: the family of open balls in $L^1(\R^3)$ of radius $\epsilon$ and centered in functions $\phi\in C^\infty_0(\R^3)$ cover the whole $L^1(\R^3)$. Finitely many of such balls cover $K$.
Therefore, we find finitely many smooth and compactly supported functions such that,
for any $\theta_0\in K$, there is at least one of such functions $\phi$ such that 
$\|u_0-\phi\|_1<\epsilon$.
But
$t^{\frac32(1-1/q)}\|e^{t\Delta}\theta_0\|_q \le \epsilon + t^{\frac32(1-1/q)}\|e^{t\Delta}\phi\|_q$.
Hence, there exists $t_0>0$ depending on $\epsilon>0$ and $K$, but not on $\theta_0$, such that
for all $0\le t<t_0$ the above expression is less than $2\epsilon$.
This yields~\eqref{eq:uniflim}.

%Applying this remark to $1<q_1<q<q_2\le\infty$ we see, by interpolation, that, as $t\to0$,
%\[
%\sup_{\theta_0\in K} t^{\frac32(1-1/q)}\|e^{t\Delta}\theta_0\|_{L^{q,\infty}}\to0.
%\]
In the same way, we obtain, as $t\to0$,
\[
\sup_{u_0\in H,\;\theta_0\in K}
\Bigl( t^{\frac12(1-3/p)}\|e^{t\Delta}[u_0+t\P\theta_0e_3]\,\|_p\Bigr)\to0,
\]
for all $1<p\le\infty$.
Now fix $3/2<q<3$ and $p>3$, such that $\frac23<\frac1p+\frac1q$.
Hence, there exists  $\eta=\eta(H,K)>0$
such that, multiplying ${\bf v_0}$ by the indicator function of the time interval $[0,\eta]$, we get, for all 
$(u_0,\theta_0)\in H\times K$,
\[\|{\bf v_0}\,1_{[0,\eta]}\|_E<1/(4C_0).\]
 The notations here are the same as in~\eqref{defa} and~\eqref{data-a},
and $E=X_p\times Y_q$.
We deduce from Theorem~\ref{theo:exi} that the solutions $(R(t)u_0,S(t)\theta_0)$, when $(u_0,\theta_0)$ vary
in $H\times K$,
are all defined at least on the time interval $[0,\eta]$.
Moreover, owing to~\eqref{poin},
\begin{equation}
 \label{preca}
 \begin{split}
  & \sup_{u_0\in H} t^{\frac12(1-3/p)}\|R(t)u_0\|_p\to0,\\
  & \sup_{\theta_0\in K} t^{\frac32(1-1/q)}\|S(t)\theta_0\|_q\to0.
  \end{split}
\end{equation}
\end{remark}

\subsection{The end of the proof of~Theorem~\ref{th:uni2}}

\begin{proof}[Proof of Theorem~\ref{th:uni2}]
The proof relies on  Proposition~\ref{th:exi}, on the uniqueness Theorem~\ref{th:uni}
and on the adaptation of ideas introduced by Ben Artzi~\cite{BenA94}, already rivisited in~\eqref{preca},
and H.Brezis~\cite{Bre94} in the 
context of the two-dimensional vorticity equation.

Without loss of generality, we can assume $3/2<q<3$.
As in Theorem~\ref{th:uni}, it is sufficient to prove that, given two solutions
$(u,\theta)$ and $(\tilde u,\tilde\theta)$, these agree on some small time interval $[0,\delta]$, for some 
$\delta>0$.

Let $R(t)u_0\colon L^3(\R^3)\to L^3(\R^3)$ and $S(t)\theta_0\colon L^1(\R^3)\to L^1(\R^3)$
the two semigroups of the solution to~\eqref{IE} constructed in Proposition~\ref{th:exi}.
The proof will consist in showing that, if $(u,\theta)$ is a solution satisfying the conditions in Theorem~\ref{th:uni2},
then $u(t)=R(t)u_0$ and $\theta(t)=S(t)\theta_0$.
%We set $\tilde u(\cdot,t)=R(t)u_0$ and $\tilde \theta(\cdot,t)=R(t)\theta_0$.

Let $0<s<\delta$, where $0<\delta<T/2$.
Recalling that $\theta$ is assumed to be locally bounded on $(0,T)$ with values in $L^{q,\infty}$,
we see that, denoting as usual $\sigma=\frac32(1-1/q)$,
\begin{subequations}
\begin{equation}
\sup_{t\in[0,\delta]} t^\sigma\|\theta(t+s)\|_{L^{q,\infty}}<\infty,
\qquad\text{and}\qquad
\lim_{t\to0} t^\sigma\|\theta(t+s)\|_{L^{q,\infty}}=0.
\end{equation}
On the other hand, $(u(s),\theta(s))\in L^3(\R^3)\times L^1(\R^3)$ and, by Proposition~\ref{th:exi}, we have
\begin{equation}
\label{disa}
\sup_{t\in[0,\delta]} t^\sigma\|S(t)\theta(s)\|_{L^q}<\infty,
\qquad\text{and}\qquad
\lim_{t\to0} t^\sigma\|S(t)\theta(s)\|_{L^q}=0.
\end{equation}
\end{subequations}
Of course, \eqref{disa}, remains true if we use the weak norm $L^{q,\infty}$ instead of the usual $L^q$-norm.
This allows us to apply the uniqueness result of Theorem~\ref{th:uni}
with the initial data
$(u(s),\theta(s))$ and in the time interval $[s,s+\delta]$: hence, 
\[u(t+s)=R(t)u(s),\qquad\text{and}\qquad 
\theta(t+s)=S(t)\theta(s),\]
for all $0<s<\delta$ and all $t\in [0,\delta]$.

Let $K=\theta((0,T])$. Notice that $K$ is precompact in $L^1(\R^3)$, because $\theta$ is continuous from  $[0,T]$ with values in $L^1(\R^3)$.
But then
\begin{equation}
\label{gitt}
\begin{split}
\sup_{s\in(0,\delta)} t^\sigma\|\theta(t+s)\|_{L^q}
&=\sup_{s\in(0,\delta)} t^\sigma \|S(t)\theta(s)\|_{L^q}\\
%&=\sup_{\theta(s) \in K} t^\sigma \|S(t)\theta(s)\|_{L^q}\\
&\le\sup_{\theta_0\in K} t^\sigma\|S(t)\theta_0\|_q\to0.
\qquad \text{as $t\to0$}.
\end{split}
\end{equation} 
Indeed, we applied here Remark~\ref{rem:time}, and more specifically~\eqref{preca}.

From~\eqref{gitt} we get $\lim_{t\to0} t^\sigma\|\theta(t)\|_{L^q}=0$.
So, in particular, $\lim_{t\to0} t^\sigma\|\theta(t)\|_{L^{q,\infty}}=0$.
More precisely, we have 
\[
\lim_{t\to0} t^\sigma\|\theta(t)\|_{L^{q,\infty}}=0.
\qquad \text{and}\qquad \theta\in Y_{q,\infty},
\]
owing to the assumption $\theta\in L^\infty_{\rm loc}((0,T),L^{q,\infty}(\R^3))$.

On the other hand, we also know by Theorem~\ref{theo:exi} that
\[
\lim_{t\to0} t^\sigma\|S(t)\theta_0\|_{L^q}=0
\qquad
\text{and}\qquad
S(t)\theta_0\in Y_q.
\]
Moreover, both $u$ and $R(t)u_0$ are in $C([0,T],L^3(\R^3))$.
This is more than needed to apply the uniqueness Theorem~\ref{th:uni}.
Hence, $u(t)=R(t)u_0$ and $\theta(t)=S(t)\theta_0$.
\end{proof}

%%%%%%

%{\tt ------------------ Possible variants/developements ---------------------
%
%\medskip
%
%- In Lemma~\ref{lem:b3}, 
%$q=3/2$ seems not to be acceptable because  $u\theta$ would belong to $L^{1,\infty}$ which
%is only a quasi-normed space. However, everything may be be OK if we work $L^{3/2,1}$, instead of %$L^{3/2,\infty}$. (could be difficult ?)
%
%- Give other variants of uniqueness theorem: e.g. putting more restrictions on the function space %of $u$ and less restrictions on the function space of $\theta$.
%
%- Give example of distributions in the closure of the $B^{-3(1-1/q)}_{q,\infty}$.
%
%- discuss global existence with small data in Besov spaces.
%}

%%%%%%%%%%%%%%%%%%%%

\end{document}